\documentclass[11pt,reqno]{amsart}

\usepackage[utf8]{inputenc}
\usepackage[theoremdefs,final]{latexdev}

\usepackage{geometry}                % See geometry.pdf to learn the layout options. There are lots.
\usepackage[numbers]{natbib}
\usepackage{todonotes}
\DeclareMathOperator{\divv}{div}

\newcommand{\pair}[1]{\left\langle #1 \right\rangle}

\providecommand{\norm}[1]{\lVert#1\rVert}
\providecommand{\abs}[1]{\lvert#1\rvert}

\newcommand{\RR}{{\mathbb R}}

\newcommand{\vol}{\mu_g}

\newcommand{\Diff}{\mathrm{Diff}}
\newcommand{\Xcal}{\mathfrak{X}}
\newcommand{\Diffvol}{{\Diff_{\vol}}}
\newcommand{\Xcalvol}{{\Xcal_{\vol}}}

\newcommand{\Prob}{\mathrm{Prob}}
\newcommand{\Jac}{\mathrm{Jac}}

\newcommand*\id{\mathrm{id}}

\def\ph{\varphi}

\let\on=\operatorname

\title[On Geodesic Completeness for Metrics on Probability Densities]{On Geodesic Completeness for Riemannian Metrics on Smooth Probability Densities}

\author[Bauer]{Martin Bauer}
\address[M.\ Bauer]{Department of Mathematics, Florida State University}
\email{bauer@math.fsu.edu}

\author[Joshi]{Sarang Joshi}
\address[S.\ Joshi]{Department of Bioengineering, Scientific Computing and Imaging Institute, University of Utah}
\email{sjoshi@sci.utah.edu}

\author[Modin]{Klas Modin}
\address[K.\ Modin]{Department of Mathematical Sciences, Chalmers University of Technology and the University of Gothenburg}
\email{klas.modin@chalmers.se}

\date{\today}                                           % Activate to display a given date or no date

\begin{document}

\begin{abstract}
	The geometric approach to optimal transport and information theory has triggered the interpretation of probability densities as an infinite-dimensional Riemannian manifold.
	The most studied Riemannian structures are the Otto metric, yielding the $L^2$-Wasserstein distance of optimal mass transport, and the Fisher--Rao metric, predominant in the theory of information geometry.
	On the space of smooth probability densities, none of these Riemannian metrics are geodesically complete---a property desirable for example in imaging applications. 
	That is, the existence interval for solutions to the geodesic flow equations cannot be extended to the whole real line.
	Here we study a class of Hamilton--Jacobi-like partial differential equations arising as geodesic flow equations for higher-order Sobolev type metrics on the space of smooth probability densities.
	We give order conditions for global existence and uniqueness, thereby providing geodesic completeness.
	The system we study is an interesting example of a flow equation with loss of derivatives, which is well-posed in the smooth category, yet non-parabolic and fully non-linear.
	On a more general note, the paper establishes a link between geometric analysis on the space of probability densities and analysis of Euler-Arnold equations in topological hydrodynamics.
% Of concern is the study of higher order metrics on the space of smooth probability densities. Similar as the Fisher-Rao metric - prevalent in the field of information geometry - and the Wasserstein-metric, which is the central object in the field of optimal mass transport, the metrics studied in this article arise via projection from Riemannian metrics  on the diffeomorphism group. Using deep results on right-invariant metrics on this infinite-dimensional Fr\'echet Lie group, we are able to obtain local and global well posedness results for the corresponding geodesic equations on the space of smooth probability densities. The investigations in these article are further motivated by connections to image analysis and recent work on regularized optimal transport metrics.   
\end{abstract}

\maketitle

\tableofcontents

% \newpage

\section{Introduction}

The space of probability densities on a manifold $M$ has an eminent rôle in several fields of mathematics.
In addition to statistics, it is central in the Monge formulation of optimal mass transport.
The work of \citet{Br1991} on the $L^2$ formulation of optimal transport, and of \citet{BeBr2000} and \citet{Ot2001} on the (formal) Riemannian structure of the $L^2$-Wasserstein distance, has fueled a geometric approach to the space of probability densities as an infinite-dimensional Riemannian manifold.
This view is further emphasized in the field of information geometry~\cite{AmNa2000,Fr1991,KhLeMiPr2013,Am2016}, which aims to describe statistics and information theory in the language of differential geometry, thereby relating statistical and geometrical concepts. 

Riemannian structures on the space of probability densities come in two dominant flavors.
In optimal transport, we have the \emph{Otto metric}, which is of Sobolev type $H^{-1}$. 
% At the tangent space of a probability function $\rho$, 
At each tangent space, it is formally defined in terms of the inverse of a second order elliptic operator depending on the base point (see \autoref{sub:summary} below).
To make Otto's geometry rigorous, one usually works on the metric space $\mathcal{P}_2(M)$ of Borel probability measures with finite second moments, equipped with the $L^2$-Wasserstein distance.
Gradient flows with respect to the Otto metric can then be handled through De Giorgi's minimizing movements~\cite{Gi1993} in the setting of metric geometry, as developed by~\citet*{AmGiSa2008}.
% \todo[author=KM]{Say something more about rigorouos Otto calculus.}\
% $\nabla\cdot\rho\nabla$ acting on functions of zero mean 

The second flavor occurs in information geometry.
Here, the Riemannian structure is given by the \emph{Fisher--Rao metric}, which is of type $L^2$.
More precisely, by a change of coordinates induced by the square root map, the Fisher--Rao geometry of the space of probability densities corresponds to an open subset of the infinite-dimensional $L^2$-sphere~\cite{Fr1991}. 
Thus, the geometry associated with the Fisher--Rao metric is fully known: the curvature is constant, positive, and the geodesics are given by great circles.
Contrary to the Otto metric, the Fisher--Rao metric is canonical, meaning it does not depend on the Riemannian structure of~$M$.
Consequently, the Fisher--Rao metric is invariant with respect to any choice of coordinates, or, equivalently, it is invariant under the action of the group of diffeomorphisms.
Up to a multiple it is, in fact, the only Riemannian metric on the space of smooth probability densities with this property~\cite{AJLV2015,BaBrMi2015}.
% Both the Otto and 
% When solving the geodesic boundary value problem, th
The Fisher--Rao metric renders the space of probability densities \emph{geodesically convex}, an essential property for the geodesic boundary value problem.
Thus, any two probability densities are connected by a unique minimizing Fisher--Rao geodesic.
This follows directly from the simple sphere geometry.
For the Otto metric, the situation is more complicated:
on $\mathcal{P}_2(M)$ we have geodesic convexity, but on the subspace of smooth probability densities geodesic convexity depends on the curvature of $M$ (for details, see \cite[\S\!~12]{Vi2009}).

Recently, several authors  combined the Otto and the Fisher--Rao metric to obtain a mixed order metric on the space of \emph{unbalanced measures}, i.e., measures that do not have the same total mass. 
This work is motivated by imaging applications \cite{ChScPeVi2016}, gradient flow problems \cite{LiMiSa2016,KoMoVo2016}, and by mass transport problems for contact structures \cite{Rezakhanlou2016}. 
In contrast to these Riemannian structures, the paper at hand considers optimal mass transport metrics induced by arbitrary order, right-invariant metrics on the diffeomorphism group.\todo{Also review other work on regularized optimal transport.}

% This 
When studying geodesic initial value problems, one is often interested in \emph{geodesic completeness}, i.e., the question of unbounded existence interval for geodesics. 
As an example, the long-standing problem of global existence of the incompressible Euler equations is, by Arnold's geometric interpretation~\cite{Ar1966}, a question of geodesic completeness.
In the category of smooth probability densities, none of the Otto or Fisher--Rao geometries are geodesically complete. %, which is of interest for the geodesic initial value problem. 
Thus, geodesic curves with respect to these geometries have bounded existence intervals. 
This has triggered the search for geodesically complete Riemannian structures on the space of smooth probability densities.
The natural approach is to look for higher-order $H^k$-type metrics.
The search for such metrics is further motivated by applications in mathematical imaging; here one is interested in finding optimal deformations between densities representing data from applications. 
Using a Riemannian metric on the diffeomorphism group to encode the differences between the shapes yields the framework of `Large Deformation Diffeomorphic Metric Matching'~\cite{BeMiTrYo2005,MiTrYo2002}. 
% It is woth to note here, that 
% the construction used in this article to induce a Riemannian metric on the space of densities is excatly the underlying concept of this framework. 

In this paper we study geodesic equations for Riemannian $H^k$-type metrics on the space of smooth probability densities.
The geodesic equation is a non-linear PDE of order $2k$, closely related to the Hamilton--Jacobi equation of fluid dynamics (it can be seen as a regularized version).
We prove geodesic completeness: if $M$ is a closed manifold of dimension $d$ and $k > d/2$, then the Cauchy problem has a unique, global solution in the Fréchet topology of smooth probability functions.
The proof is based on powerful techniques developed for geodesic equations on groups of diffeomorphisms \cite{MP2010,EK2014a,BrVi2014,BEK2015}.
When working with groups of diffeomorphism, the standard trick is to use Sobolev $H^s$ completions for $s>d/2+1$.
In our setting this is not possible: the strong form of the geodesic initial value problem is not well-posed in the $H^s$ topology.
Instead, it is essential to work in the Fréchet topology of smooth functions, as developed by \citet{Ha1982} and by \citet{Omori1997}.

\subsection{Summary of main result}\label{sub:summary}
In this section we summarize the main result. % in the simplified case when $M$ is the flat $d$-torus $\TT^d$.
We also present the relation to the Otto metric and $L^2$ optimal transport.

Let $(M,g)$ be an orientable closed Riemannian $d$-manifold.
Denote by $\vol$ the induced Riemannian volume form.
% Let $\ud x = \ud x^1\wedge\ldots\wedge \ud x^d$ be the standard volume form on $\TT^d$.
A smooth probability density is a weighted volume form $\mu= \rho\,\vol$, where $\rho\in C^\infty(M)$ is a smooth, real-valued \emph{probability function} fulfilling
\begin{equation}\label{eq:rho_condition}
	\rho(x) > 0,\quad \forall\, x\in M \qquad\text{and}\qquad \int_{M}\rho \, \vol = 1 .
\end{equation}
The set of such functions is denoted $P^\infty(M)$. %; it is an affine subspace of $C^\infty(M)$.
Its tangent elements belong to the space $C^\infty_0(M)$ of smooth functions with vanishing mean.

% The tangent space of

 % the infinite-dimensional manifold of smooth probability densities consists of functions with vanishing mean value.

% If $t\mapsto\rho(t,\cdot)$ is a differentiable curve in $P^\infty(M)$, then the time derivative $\rho_t(t,\cdot)$ integrates to zero.
% for any $t$
% \begin{equation}
% 	\int_{\TT^d} \rho_t(t,\cdot) \, \ud x = 0.
% \end{equation}
% Thus, the tangent space of the infinite-dimensional manifold of smooth probability densities consists of functions with vanishing mean value.
To specify a Riemannian metric we need to give, for each probability function $\rho\in P^\infty(M)$, an inner product on the tangent space $T_\rho P^\infty(M) \simeq C^\infty_0(M)$.
We first recall how this is done for the Otto metric.
From the fluid dynamical approach to optimal transport of \citet{BeBr2000}, we know that $\rho$ is evolving according to the continuity equation
\begin{equation}\label{eq:continuity_eq}
	\rho_t + \divv (\rho u) = 0,
\end{equation} 
where $u(t,x)$ is the time-dependent vector field governing the fluid.
For the Otto metric, one takes this vector field to be 
\begin{equation}\label{eq:u_grad_p}
	u = \nabla p	
\end{equation}
for some potential function~$p$.
Through the continuity equation \eqref{eq:continuity_eq} we then get a self-adjoint, second order elliptic equation relating $p$ to $\rho_t$.
Now, the Otto Riemannian metric at the base point $\rho$ is given by
\begin{equation}\label{eq:general_p_metric}
	\pair{\rho_t,\rho_t}_\rho = \int_{M} \rho_t p \, \vol.
\end{equation}

The metric studied in this paper is very similar to that of Otto: it is of the form \eqref{eq:general_p_metric} with $\rho_t$ fulfilling the continuity equation \eqref{eq:continuity_eq}.
The difference is the relation between $p$ and the fluid vector field $u$; instead of \eqref{eq:u_grad_p} we take
\begin{equation}\label{eq:u_regularized}
	(1-\Delta)^{k+1} u = \rho\nabla p , \quad k\geq -1,
\end{equation}
where $\Delta$ denotes the Laplace--Beltrami operator.
% That is the difference.
Notice that the operator $(1-\Delta)^{k+1}$ has a regularizing effect on the vector field $u$.
This regularization is, essentially, the key ingredient that leads to geodesic completeness.

Next we consider the geodesic equation.
For the Otto metric, it is given by the continuity equation \eqref{eq:continuity_eq} coupled with  the Hamilton--Jacobi equation
\begin{equation}\label{eq:ham_jacobi}
	p_t + \abs{\nabla p}^2 = 0 .
\end{equation}
More precisely, the equations \eqref{eq:continuity_eq} and \eqref{eq:ham_jacobi} constitute the \emph{Hamiltonian form} of the geodesic equation for the Otto metric: $\rho$ is the configuration and $p$ is the momentum variable.

% for the kinetic energy Hamiltonian $H(q,p) = \abs{p}^2/2$.
% More precisely, this is the 
% In our case, the geodesic equation

The Hamiltonian form of the geodesic equation for the regularized metric studied in this article 
%
% \begin{equation}
% 	\pair{\rho_t,\rho_t}_\rho = \int_{\TT^d} \rho_t p \, \ud x,
% \end{equation}
% where $p\in C^\infty(\TT^d)$ is related to $\rho_t$ through the self-adjoint elliptic pseudo-differential equation
% \begin{equation}
% 	\rho_t = -\nabla\cdot (\rho (1-\Delta)^{-k} (\rho \nabla p)).
% \end{equation}
% We stress here the resemblance with the Otto metric, for which the relation between $p$ and $\rho_t$ is $\dot\rho = -\nabla\cdot\rho \nabla p$.
% To write down the geodesic equation for the Riemannian metric $\pair{\cdot,\cdot}_\rho$, we first introduce a smooth, time-dependent vector field $u(t,x) \coloneqq (u^1(t,x),\ldots,u^d(t,x))$.
% The geodesic equation then 
%
consists in 
%finding $\rho,p,u^i \in C^\infty(\RR\times\TT^d)$ that fulfill 
the system of partial differential equations given by
\begin{equation}\label{eq:geodesic_pde_torus}
	\begin{split}
		(1-\Delta)^{k+1} u - \rho\nabla p &= 0, \\
		p_t + u\cdot \nabla p &= 0, \quad p(0,\cdot) = p_0 \\
		\rho_t + \divv(\rho u) &= 0, \quad \rho(0,\cdot) = \rho_0	.
	\end{split}
\end{equation}
(See \autoref{thm:main:result} for a derivation.)
Notice, as expected, the close relation to the Hamilton--Jacobi equation: if $\rho \equiv 1$ and $k=-1$, then $p$ in \eqref{eq:geodesic_pde_torus} fulfills the Hamilton--Jacobi equation~\eqref{eq:ham_jacobi}.
Furthermore, from the continuity equation it follows that
\begin{equation}
	\int_{M}\rho_0\, \vol = \int_{M} \rho(t,\cdot)\,\vol .
\end{equation}
Thus, total mass (or probability) is conserved. % for solutions of \eqref{eq:geodesic_pde_torus}.
% Thus, solutions to \eqref{eq:geodesic_pde_torus} evolve on the

% if the initial density $\rho_0 \, \ud x$ has unitary volume, then so has $\rho(t,\cdot)\ud x$ for any~$t$.

We are now ready to state the main result.

\begin{theorem}\label{thm:main result}
	Let $k > d/2$, $\rho_0 \in P^\infty(M)$, and $p_0 \in C_0^\infty(M)$.
	Then:
	\begin{enumerate}
		\item There exist functions $\rho,p \in C^\infty(\RR\times M)$ and a vector field $u\in C^\infty(\RR\times M,TM)$ fulfilling equation \eqref{eq:geodesic_pde_torus}.
		The solution is unique up to addition of constants to $p$.
		 % and the condition $p(t,\cdot) \in C_0(M)$ for all $t\in\RR$;
		% \begin{equation}
		% 	% \rho(t,\cdot) \in P^\infty(M). %, \qquad 
		% 	p(t,\cdot) \in C_0(M)
		% \end{equation}

		\item For all $t\in\RR$ the solution fulfills $\rho(t,\cdot) \in P^\infty(M)$.
		% \begin{equation}
		% 	\rho(t,\cdot) \in P^\infty(M) %, \qquad 
		% 	% p(t,\cdot) \in C_0(M).
		% \end{equation}

		\item For any fixed $t\in\RR$, the solution map
		\begin{equation}
			(\rho_0,p_0)\longmapsto \Big(\rho(t,\cdot),p(t,\cdot),u(t,\cdot)\Big) 
			\quad \text{with}\quad p(t,\cdot)\in C_0^\infty(M)
		\end{equation}
		is smooth (with respect to the standard Fréchet topology of $C^\infty(M)$).

		% $\rho(t,\cdot),p(t,\cdot),u^1(t,\cdot),\ldots,u^d(t,\cdot)$  depends smoothly  on the initial data $\rho_0$ and $p_0$.
	\end{enumerate}
	% If $k\geq d/2$ and $\rho_0,p_0 \in C^\infty(M)$, where $\rho_0$ fulfills \eqref{eq:rho_condition}
	% Then there exists unique functions $\rho,p,u^1,\ldots,u^d \in C^\infty(\RR\timesM)$ fulfilling equation \eqref{eq:geodesic_pde_torus} and, for all $t\in\RR$,
	% \begin{equation}
	% 	\rho(t,\cdot) \in P^\infty(M), \qquad p(t,\cdot) \in C_0(M).
	% \end{equation}
	Consequently, the space $P^\infty(M)$ of smooth probability functions, equipped with the Riemannian metric \eqref{eq:general_p_metric} defined by the equations \eqref{eq:continuity_eq} and \eqref{eq:u_regularized}, is geodesically complete.
	% Furthermore, the solution $\rho,p,u^1,\ldots,u^d \in C^\infty(\RR\times\TT^d)$ depends smoothly on the initial data $\rho_0$ and $p_0$.
\end{theorem}

The remaining parts of the paper consists in proving this result in the slightly generalized case when equation \eqref{eq:u_regularized} is replaced by a more general elliptic system. 
% In \autoref{} we will also prove 
We also give a local well-posedness result under the weaker condition $k\geq 0$.
%We have organized the proof as follows...\todo{Finish this.}

% \begin{theorem}
% Consider the Sobolev type metric of order $k$ on the space of smooth probability densities, given by:
% $$\bar G^k_{\mu}(\dot \rho \vol,\dot \rho \vol)=\int_M  (\bar A_\rho \dot \rho)\; \dot\rho \;\vol, \qquad\text{ with }\qquad \bar A_\rho= -\left(\on{div} \rho (1+\Delta)^{-k-1} \rho\nabla \right)^{-1}\;.$$
% The corresponding geodesic equations are given by:
% \begin{equation}\label{geodesic:equation}
%  \begin{aligned}
% p &= \bar A_{\rho}(\rho_t),\\
% p_t&=-(\nabla p)^t.(1+\Delta)^{-k-1}(\rho \nabla p)
% \end{aligned}
% \end{equation}
% If $k\geq 0$, then there exists for any initial conditions a unique non-extendable solution $ (\mu,\al)$ of \autoref{geodesic:equation} and the solution depends smoothly on the initial data. For any $k\geq \frac{m}{2}$ the maximal interval of existence is equal to $\mathbb R$, i.e., 
% the space $(\Prob(M),\bar G^k)$ is  geodesically complete.
% \end{theorem}

\section{Spaces of diffeomorphisms and densities}\label{sec:diff_and_dens}
Let $M$ be a smooth, compact, and orientable Riemannian $d$-manifold without boundary.
The metric tensor is denoted by $g$ and the corresponding volume form  by $\vol$.
Without loss of generality, we may assume that $\int_M\vol = 1$.
The space of smooth $k$-forms on~$M$ is denoted $\Omega^k(M)$.

Of central interest in this paper is the space of smooth \emph{probability densities}
\begin{equation}
	\Prob(M) = \{ \mu\in\Omega^d(M)\mid \mu>0, \quad \int_M\mu = 1 \}.
\end{equation}
This space is naturally equipped with an infinite-dimensional Fréchet topology, making it a \emph{Fréchet manifold} \cite[\S\!~III.4.5]{Ha1982}.
Its tangent bundle is thereby also a Fréchet manifold, and the tangent spaces are given by
\begin{equation}
	T_\mu\Prob(M) = \{ \alpha\in\Omega^d(M)\mid \int_M\alpha = 0 \}.
\end{equation}
Using the Riemannian volume form $\vol$, we can identify $\Prob(M)$ with the space of smooth probability functions $P^\infty(M)$.
% {\color{red} Should we define   $P^\infty(M)$ again? So far we introduced it only in the introduction.}
Indeed, throughout this paper we implicitly associate to each $\mu\in\Prob(M)$ a corresponding $\rho\in P^\infty(M)$ through the relation
\begin{equation}
	\mu = \rho\,\vol .
\end{equation}
Thus, $\rho$ is used in place of $\mu$ whenever convenient.

\subsection{Action of diffeomorphisms on densities}
The diffeomorphism group $\Diff(M)$ is an infinite-dimensional Fréchet Lie group, i.e., it is a Fréchet manifold and the group operations (composition and inversion) are smooth maps \cite[\S\!~I.4.6]{Ha1982}.
The corresponding Fréchet Lie algebra is the space $\Xcal(M)$ of smooth vector fields equipped with minus the vector field bracket.
% \todo{Martin: Its the negative of the vector field bracket?}

$\Diff(M)$ acts on $\Prob(M)$ from the left by pushforward and from the right by pullback
% The right action is given by pullback:
\begin{align}
\Diff(M)\times \Prob(M)&\mapsto \Prob(M)\\
(\varphi,\mu)&\rightarrow \varphi_*\mu
\end{align}
\begin{align}
\Diff(M)\times \Prob(M)&\mapsto \Prob(M)\\
(\varphi,\mu)&\rightarrow \varphi^*\mu\;.
\end{align}
The corresponding actions on $P^\infty(M)$ are
\begin{align}
\Diff(M)\times P^\infty(M)&\mapsto P^\infty(M)\\
(\varphi,\rho)&\rightarrow \Jac(\varphi^{-1})\rho\circ\varphi^{-1}
\end{align}
\begin{align}
\Diff(M)\times P^\infty(M)&\mapsto P^\infty(M)\\
(\varphi,\rho)&\rightarrow \Jac(\varphi)\rho\circ\varphi\;.
\end{align}
By a result of \citet{Mo1965}, these actions are transitive.

The left and right action of $\Diff(M)$ on the fixed volume element $\vol$ yield projections $\Diff(M)\to\Prob(M)$.
Later we shall need the following result of Hamilton.
% Mosers lemma these actions are transitive if restricted to the space of probability densities, i.e.,
% for each $\mu,\nu\in \Prob(M)$, there exists a $\varphi\in \Diff(M)$ such that 
% $\varphi_*\mu=(\varphi^{-1})^*\mu=\nu$, see \cite{Mo1965}. Thus we obtain the result:
\begin{theorem}\label{thm:principal_bundle_diff_dens}
The set of volume preserving diffeomorphisms
\begin{equation}
	\Diffvol(M) = \{\varphi\in\Diff(M)\mid \varphi^*\vol = \vol \}
\end{equation}
is a closed Fréchet Lie subgroup.
Furthermore, the projection maps
\begin{align}
\pi_l: \on{Diff}(M)&\mapsto \Prob(M)\\
\varphi&\mapsto \varphi_*\vol\;,
\end{align}
and 
\begin{align}
\pi_r: \on{Diff}(M)&\mapsto \Prob(M)\\
\varphi&\mapsto \varphi^*\vol\;,
\end{align}
are smooth principal $\Diffvol(M)$-bundles over $\Prob(M)$ with respect to the right and left action of $\Diffvol(M)$ on $\Diff(M)$.
Hence, the set of left cosets $$\Diff(M)/\Diffvol(M)$$ is identified with $\Prob(M)$ by $\pi_l$, and the set of right cosets $$\Diffvol(M)\backslash\Diff(M)$$ is identified with $\Prob(M)$ by $\pi_r$.
\end{theorem}

\begin{proof}
	That $\Diffvol(M)$ is a closed Fréchet Lie subgroup of $\Diff(M)$, and that $\pi^r$ defines a smooth principal bundle is proved by \citet[Th.~III.2.5.3]{Ha1982} using the Nash--Moser inverse function theorem.
	That $\pi_l$ also defines a smooth principal bundle follows since the inversion on $\Diff(M)$ is smooth and since $\Diffvol(M)$ is a Lie subgroup.
\end{proof}

% \begin{remark}
%  Identifying densities and functions as above the projection $\pi_l$ is given by
% \begin{align}
% \pi_l: \on{Diff}(M)&\mapsto \Prob(M)\cong P^{\infty}(M) \\
% \varphi&\mapsto \operatorname{Jac}(\varphi^{-1})\;.
% \end{align}
% \end{remark}

For both projections $\pi_l$ and $\pi_r$ we can calculate the corresponding vertical bundles, defined by the kernel of the tangent mapping.
\begin{lemma}\label{lem:vertical_bundle}
The vertical bundles of the projections $\pi_l$ and $\pi_r$ are given by
\begin{align}
	\operatorname{Ver}^l_{\varphi}
	&=\left\{\dot\varphi \in T_\varphi\Diff(M)\mid  \operatorname{div}(\rho u) = 0, \; u\coloneqq \dot\varphi\circ\varphi^{-1},\; \rho \coloneqq \Jac(\varphi^{-1}) \right\}.
	\\
	\operatorname{Ver}^r_{\varphi}
	&=\left\{\dot\varphi \in T_\varphi\Diff(M)\mid \operatorname{div}(u)=0 , \; u\coloneqq\dot\varphi\circ\varphi^{-1}\right\}.
\end{align}
% Here $\operatorname{div}^{\mu}(X)=\frac{\mathcal L_X \mu}{\mu}%=\mathcal L_{X}\mu
% $ denotes the divergence of the vector field $X$ with respect to the volume 
% form $\mu$. In the latter we will use sometimes $\operatorname{div}(X)$ as an abbreviation for the divergence with respect to the standard density $\vol$, i.e., $\operatorname{div}(X)=\operatorname{div}^{\vol}(X)$.
\end{lemma} 

\begin{proof}
To  calculate the differential of the projection mappings let $\phi(t,\cdot)$ be a path of diffeomorphisms  with 
\begin{align}
\phi(0,\cdot)&=\varphi\\
\partial_t\big|_{t=0}\phi &= h := u \circ\varphi \text{ for some } u\in \Xcal(M).                                                            
\end{align}
The derivative of the right projection can be calculated via  
\begin{align}
T_{\varphi}\pi_r(u\circ\varphi)=\partial_t\big|_{t=0}\left(\phi(t)^*\vol\right) = \varphi^* \mathcal L_u \vol= \varphi^*\left( \on{div}(u)\vol\right)
\end{align}
which vanishes if and only if $u$ is divergence free with respect to $\vol$.  

For the derivative of the left projection we use 
\begin{align}
0= \partial_t\big|_{t=0}\left(\phi(t)^*\phi(t)_*\vol\right)=\phi(t)^*\left(\mathcal L_u \varphi_*\vol \right)+
\phi(t)^*\partial_t\big|_{t=0}\left( \varphi_*\vol \right)
\end{align}
to obtain
\begin{align}\label{derivative_pi_l}
T_{\varphi}\pi_l(u\circ\varphi)=\partial_t\big|_{t=0}\left( \phi_*\vol \right)=-\mathcal L_u \varphi_*\vol
% =-\on{div}^{\varphi_*\vol}(u)\varphi_*\vol
= -\on{div}(\rho u)\vol
\end{align}
where $\rho=\on{Jac}(\varphi^{-1})$. % and $\divv^\mu$ denotes the divergence with respect to $\mu$.
\end{proof}

\subsection{Spaces of Sobolev diffeomorphisms and densities}
To obtain results on existence of geodesics on $\Diff(M)$, the standard approach is to work in the Banach topology of Sobolev completions, and then use a `no-loss-no-gain' in regularity result by \citet{EbMa1970}.
% manifolds of mappings it is often convenient to study the metric on the level of Sobolev completions. This allows one  
% to prove the well posedness in the smooth category using a no-loss-no-gain in regularity result \cite{EbMa1970}. 
In this section we examine the extensions of the left and right projections in \autoref{thm:principal_bundle_diff_dens} to the Sobolev category.
While the right projection extends smoothly, the left projection turns out to be only continuous. 
Essentially for this reason, the main result of our paper (\autoref{thm:main result}) is valid for smooth probability densities, but not for probability densities with finite regularity.

We first introduce the group of Sobolev diffeomorphisms
\begin{align}
\mathcal{D}^s(M)=\left\{ \varphi\in H^s(M,M)\mid \varphi \text{ is bijective and } \varphi^{-1}\in H^s(M,M)\right\},\quad s>\frac{d}{2}+1\;,
\end{align}
which is a Hilbert manifold and a topological group.
It is, however, not a Lie group, since left multiplication is not smooth (only continuous).
The corresponding set of Sobolev vector fields is denoted $\Xcal^s(M)$. 
For a detailed treatment on these groups we refer to the research monograph by \citet*{IKT2013}.

Similarly we consider the space of Sobolev probability densities
\begin{align}
 \Prob^{s}(M)&=\{\rho\,\vol\mid \rho \in H^{s}(M,\mathbb R_{+}) \text{ with } \int_M \rho\, \vol =1 \text{ and } \rho>0 \},\quad s>\frac{d}{2}\;.
\end{align}
Note that we need $s>\frac{d}{2}$ to ensure that the condition $\rho>0$ is well defined point-wise.

The projections in \autoref{thm:principal_bundle_diff_dens} extend to the Sobolev completions as follows.
\begin{lemma}\label{lem:vertical}
Let $s>\frac{d}{2}+1$ and let $\pi_l$ and $\pi_r$ be the left and right projections in \autoref{thm:principal_bundle_diff_dens}.
\begin{enumerate}
\item  The left projection $\pi_l$ extends to a continuous surjective mapping
\begin{equation}\label{projection_left}
\begin{aligned}
\pi^s_l: \mathcal{D}^s(M)&\mapsto \Prob^{s-1}(M)\\
\varphi&\mapsto \varphi_*\vol\;.
\end{aligned}
\end{equation}
For any $s<\infty$ this mapping is not $C^1$.
% it is $C^0$, but never $C^1$, since $T\pi^q$ has only values in $T\Prob^{q-2}(M)$ and not in
% $T\Prob^{q-1}(M)$. 
\item 
The right projection $\pi_r$ extends to a smooth surjective mapping
\begin{equation}\label{projection}
\begin{aligned}
\pi^s_r: \mathcal{D}^s(M)&\mapsto \Prob^{s-1}(M)\\
\varphi&\mapsto \varphi^*\vol\;.
\end{aligned}
\end{equation}
\end{enumerate}

\end{lemma}
\begin{proof}
That $\pi_r^s$ is a smooth surjective mapping is proved by \citet{Eb1970}: one can see from the formula
\begin{align}
T_{\varphi}\pi_r(u\circ\varphi)= \varphi^*\left( \on{div}(u)\vol\right) = \on{div}(u)\circ\varphi\; \Jac(\varphi)
\end{align}
that there is only one loss of derivatives.
That $\pi_l^s$ is a continuous surjective mapping then follows since inversion $\varphi\to\varphi^{-1}$ is continuous in the Sobolev category.
To see that $\pi^s_l$ is not $C^1$ we recall the formula for its derivative 
\begin{align}
T_{\varphi}\pi_l(u\circ\varphi)= -\on{div}(\rho u)\vol, \quad \rho\coloneqq\on{Jac}(\varphi^{-1}).
\end{align}
% with $\rho=\on{Jac}(\varphi^{-1})$.
Since $\rho$ is of class $H^{s-1}$ the function $\on{div}(\rho X)$ is in general only of class $H^{s-2}$ which means there is a loss of two derivatives.
% This yields the desired statement.
\end{proof}

\section{Right-invariant metrics on diffeomorphisms} \label{Ap:rightinvariant}
% Let $g$ be a Riemannian metric on $M$, such that $\operatorname{vol}(g)=\vol$. 
In this section we review existence results for geodesics equations on $\Diff(M)$ with respect to right-invariant Riemannian metrics.

To define a right-invariant Riemannian metric on $\Diff(M)$ we introduce the so-called \emph{inertia operator} 
$A\colon \Xcal(M)\to \Xcal(M)$. 
We assume that $A$ is a strictly positive, elliptic, differential operator, that is self adjoint with respect to the $L^2$ inner product on $\Xcal(M)$.
Any such $A$ defines a inner product $G_{\operatorname{\id}}$ on  $\Xcal(M)$ via
\begin{equation}
 G_{\operatorname{\id}}(X,Y)=\int_M g\left(A X,  Y\right)\, \vol \,.
\end{equation}
We can extend this to a right-invariant metric on $\Diff(M)$ by right-translation:
\begin{equation}\label{eq:Gmet_DiffM}
G_{\varphi}(h,k)= G_{\operatorname{\id}}(h\circ\varphi^{-1},k\circ\varphi^{-1})= \int_M g\left(A (h\circ\varphi^{-1}),  k\circ\varphi^{-1}\right)\, \vol\,.
\end{equation}
% A similar definition can be also chosen to define right-invariant metrics on the Sobolev completions.

Well-posedness of the corresponding geodesic equations is an active area of research. 
The approach dates back to Ebin and Marsden \cite{EbMa1970}, who proved local well-posedness
of the geodesic initial value problem for the $L^2$-metric on the group of volume preserving diffeomorphisms, corresponding to the Euler equations of incompressible inviscid flow.
Using similar techniques, analogous results have been extended to a variety  of right-invariant metrics on groups of diffeomorphisms: Constantin and Kolev showed that the geodesic equation of the right–invariant Sobolev metrics 
of fractional order on $\Diff(S^1)$ is locally well-posed if $k\geq\frac12$ and globally well-posed if $k>\frac{d}2+1=\frac32$, \cite{EK2014, EK2014a}.
In \cite{BEK2015} this result is extended to fractional order metrics on diffeomorphism groups of $\RR^d$.

For diffeomorphism groups of a closed $d$-manifold $M$
the situation has been studied for integer order metrics: 
metrics of order one have been studied by Shkoller in \cite{Shk2000,Shk1998}.
Preston and Misiolek showed in \cite{MP2010} that the geodesic equation is locally well-posed 
for Sobolev metrics of integer order $k\geq 1$ and globally well-posed if $k>\frac{d}{2}+1$, see also \cite{EbMa1970,BHM2011,K2016}.
Recently Bruveris and Vialard \cite{BrVi2014} showed metric and geodesic completeness on the Banach manifold $\mathcal{D}^s(M)$ of Sobolev diffeomorphisms, provided that the metric is smooth and strong. 
%This is in particular true for the class of integer-order Sobolev metrics of order $q>\frac{d}{2}+1$, 
%see \cite{EbMa1970,MP2010,}. 

The following theorem collects known results about right-invariant geodesic equations on $\Diff(M)$ and $\mathcal{D}^s(M)$.
\begin{theorem}\label{thm:wellposednessDiff}
Let $A$ be a positive, elliptic, differential operator of order $2k$, self-adjoint with respect to the $L^2$ inner product on $\Xcal(M)$.
% Let $A$ be an inertia operator of order $2l$ that satisfies assumption~\ref{assumption:A} and 
Further, let $G$ be the right-invariant metric \eqref{eq:Gmet_DiffM} on $\Diff(M)$ induced by $A$. 
The geodesic equation (called \emph{EPDiff}), expressed in the right reduced variable $u=\varphi_t\circ\varphi^{-1}$, is given by
\begin{align}
u&=\ph_t\circ\ph^{-1} \label{eq:epdiff1} \\
u_t&=-A^{-1}\left\{\nabla_u (Au)+(\mathrm{div}\hspace{0,1cm}u)(Au)+(\nabla u)^\top (Au)\right\} , \label{eq:epdiff2}
\end{align}
where $(\nabla u)^\top$ is the pointwise adjoint of the mapping $w\mapsto \nabla_w u$ relative to the Riemannian metric $g$
% That is,
\begin{equation}
 g((\nabla u)^\top v,w)=g(v, \nabla_w u),\qquad \forall\, v,w\in\Xcal(M).
\end{equation}
We have the following results concerning existence of solutions to \eqref{eq:epdiff1}-\eqref{eq:epdiff2}.
\begin{enumerate}
 \item Let $s\geq 2k\geq 2$ and $s>\frac{d}{2}+1$. 
 Then for any initial conditions $(\ph_0,v_0)\in T\mathcal{D}^s(M)$, there exists a unique solution 
 \begin{equation}
  (\ph,v)\in C^{\infty}(J,T\mathcal{D}^s(M))
 \end{equation}
with $(\ph(0),v(0))=(\ph_0,v_0)$, defined on a non-empty, non-extendable interval $J$ containing zero.
For any $t\in J$, the solution $(\ph(t),v(t))$ depends smoothly on the initial data $(\ph_0,v_0)$. 
Moreover, for smooth initial data, i.e., $(\ph_0,v_0)\in T\Diff(M)$, the interval $J$ is the same for each $s$.
Consequently, the solution $(\ph,v)$ in this case belongs to $C^\infty(J\times M, TM)$ and the mapping $(\ph_0,v_0)\mapsto (\ph(t),v(t))$ is smooth in the Fréchet category of smooth maps.
\item Let $k > \frac{d}{2}+1$ and $s\geq 2k$. Then for any initial conditions $(\ph_0,v_0)\in T\mathcal{D}^s(M)$, the maximal interval of existence $J$ is equal to $\mathbb R$.
That is, the space $(\mathcal{D}^s(M),G)$ is geodesically complete. 
This result remains valid in the smooth Fréchet category.
% \item Let $l > \frac{d}{2}+1$. Then $(\mathcal D^l(M)_0,\on{dist}^G)$ is a complete metric space and any two elements can be joined by a minimizing geodesic.
\end{enumerate}
\end{theorem}

\section{Descending Riemannian metrics}
The fundamental notion in this paper is that the concept of \emph{Riemannian submersions} provide a (formal) process for obtaining Riemannian structures on the space of probability densities from invariant Riemannian structures on the space of diffeomorphisms.
In this section we review this process.
First we give the framework in the setting of finite-dimensional manifolds.
Thereafter, we show the necessary steps to obtain corresponding rigorous results in the infinite-dimensional Fréchet topology of interest.

\subsection{Riemannian submersions}\label{sub:submersions}

Let $E$ and $B$ be finite-dimensional, smooth manifolds, and let $\pi:E\rightarrow B$ be a submersion.
%  of smooth manifolds, 
% that is, $T\pi:TE \rightarrow TB$ is surjective. 
Then
$$\on{Ver}=\on{ker}(T\pi) \subset TE$$
is called the \emph{vertical subbundle}. 
If $E$ is Riemannian we can define the \emph{horizontal subbundle} as the complement of $\on{Ver}$
$$\on{Hor}=\on{Ver}^\bot \subset TE.$$
% If the horizontal bundle exists then every 
A vector $X \in TE$ can then be decomposed uniquely in vertical and horizontal components as
$$X=X^{\on{ver}}+X^{\on{hor}}$$
and the mapping 
\begin{equation}\label{sh:sub:eq1}
	T_x \pi|_{\on{Hor}_x}\colon\on{Hor}_x\rightarrow T_{\pi(x)}B
\end{equation}
is an isomorphism of vector spaces for all $x\in E$. 

If both $E$ and $B$ are Riemannian manifolds and if the mapping \eqref{sh:sub:eq1} is an isometry, 
then $\pi$ is called a \emph{Riemannian submersion}.
For Riemannian submersion we have the following theorem that connects the geometry on the 
base space with the horizontal geometry on the top space.

\begin{proposition}\label{Riemanniasubmersions}
	Consider a Riemannian submersion $\pi\colon E\rightarrow B$, 
	and let $\gamma\colon [0,1]\rightarrow E$ be a geodesic curve in $E$.
	\begin{enumerate}
		\item If $\dot \gamma(t_0)$ is horizontal at $t_0\in[0,1]$, then $\dot \gamma(t)$ is horizontal at all $t\in [0,1]$. 
		\item If $\dot \gamma(t)$ is  horizontal, then $\pi \circ \gamma(t)$ is a geodesic curve in $B$.
		% \item If every curve in $B$ can be lifted to a horizontal curve in $E$, 
		% then there is a one-to-one correspondence between curves in $B$ and horizontal curves in $E$. 
		% This implies that instead of solving the geodesic equation on $B$ one can equivalently solve
		% the equation for horizontal geodesics in $E$.
	\end{enumerate}
\end{proposition}

\begin{proof}
	See \cite[\S\!~26]{Mic2008}.	
\end{proof}

A Riemannian metric $g$ on $E$ is called \emph{descending} if it induces a Riemannian metric on~$B$.
More precisely, $g$ is descending if and only if, given two vectors $X\in \on{Hor}_x$ and $X'\in \on{Hor}_{x'}$ with $\pi(x)=\pi(x')$ and $T_x\pi\cdot X = T_{x'}\pi \cdot X'$, then $g_x(X,X) = g_{x'}(X',X')$.
The isomorphism \eqref{sh:sub:eq1} then defines a Riemannian metric $\bar g$ on $B$.
% , constructed so that $\pi$ is a Riemannian submersion.

Consider now the case when $\pi\colon E\to B$ is a principal $H$-bundle, so that $B\simeq E/H$ where $H$ is a Lie group acting on $E$.
$H$-invariant Riemannian metrics on $E$ are special cases of descending metrics.

\begin{proposition}\label{prop:descending_metric}
	Let $\pi\colon E\to B$ be a principal $H$-bundle, and let $g$ be an $H$-invariant Riemannian metric on $E$
	\begin{equation}
		g_x(X,X) = g_{x\cdot h}(X\cdot h,X\cdot h), \quad \forall\, X\in T_x E, h\in H.
	\end{equation}
	Then $g$ is descending, thereby inducing a unique Riemannian metric on $B$ such that $\pi$ is a Riemannian submersion.
\end{proposition}

% If $E$ is Riemannian with an $H$-invariant metric, then 

\subsection{Descending metrics on \texorpdfstring{$\Diff(M)$}{Diff(M)} }\label{sub:submersions_infinite}
We shall now discuss the framework in \autoref{sub:submersions}, but in the infinite-dimensional case $E=\Diff(M)$ and $B=\Prob(M)$.
In this setting, one has to carefully address some questions that are self-evident in the finite-dimensional case.
We set out from the smooth Fréchet principal bundle structure $\pi_l\colon\Diff(M)\to\Prob(M)$ (or $\pi_r\colon\Diff(M)\to\Prob(M)$) given in \autoref{thm:principal_bundle_diff_dens}.
Recall the vertical subbundle $\on{Ver}^l$ (or $\on{Ver}^r$) given in \autoref{lem:vertical_bundle}.

Given a Riemannian metric on $\Diff(M)$, the first complication is that the horizontal bundle might be empty even if the vertical bundle is a proper subbundle of the tangent bundle.
In other words, the complement of the vertical bundle might not exist.
In the setting considered here, with the action of $\Diff(M)$ on $\Prob(M)$, it is, however, often possible to obtain the complement of the vertical bundle within the smooth Fréchet category.

\begin{remark}
	An example where the complement exists only in a completion of the space is when the projection is given by the action of $\Diff(M)$ on a fixed point $x_0\in M$, i.e., $\pi^{x_0}:\Diff(M)\rightarrow M$ with $\pi^{x_0}(\varphi)=\varphi(x_0)$.
	The vertical bundle then consists of smooth vector fields vanishing at $x_0$.
	If the Riemannian metric on $\Diff(M)$ is of finite order $k$, then the horizontal bundle consists of vector fields with finite regularity at $x_0$ (the $k$-th derivative is a delta distribution).
	In this example, the horizontal bundle therefore does not exist in the smooth Fréchet category.	
\end{remark}

The second complication is that the statements of \autoref{Riemanniasubmersions} must be validated explicitly, typically by proving that the mapping \eqref{sh:sub:eq1} is an isomorphism.
One can then check that initially horizontal geodesics remain horizontal throughout their existence interval, and that such horizontal geodesics on $\Diff(M)$ project to geodesic curves on $\Prob(M)$.
% \todo{MB: The second complication appears somehow late. We mention now Fisher-Rao and Wasserstein first and only afterwards mention the second complication, which might be a bit late}

Let us now give two infinite-dimensional examples where the two complications can be addressed.
The first example is when the Riemannian structure on $\Diff(M)$ is given by the $L^2$-type metric
\begin{equation}
	G_{\varphi}(u\circ\varphi, u\circ\varphi) = \int_M g(u,u)\, \varphi_*\vol.
\end{equation}
This metric is not fully right-invariant, but it is right-invariant with respect to the subgroup $\Diffvol(M)$.
Thus, it (formally) descends to $\Prob(M)$ through the projection~$\pi_l$.
In fact, the corresponding Riemannian metric on $\Prob(M)$ is exactly the Otto metric, discussed in \autoref{sub:summary} above, which is fundamental in the dynamic and geometric approach to optimal mass transport (see the work of \citet{Ot2001} and \citet{BeBr2000}).
From the Helmholtz decomposition of smooth vector fields, it is straightforward to show that the horizontal bundle in this case is given by the right translated gradient vector fields
\begin{equation}
	\on{Hor}_\varphi = \{\nabla p\circ\varphi\mid p \in C^\infty(M) \}.
\end{equation}

There is also an example of a Riemannian metric on $\Diff(M)$ descending to $\Prob(M)$ through the right projection $\pi^r$.
Indeed, a right-invariant family of such metrics are given in \cite{Mo2014}.
In this case, the corresponding metric on $\Diff(M)$ is the Fisher--Rao metric, instrumental in information geometry.
Since right projection is used, the right-invariance is not exhausted when taking the quotient, and is therefore still present after the projection. 
Consequently, the Fisher--Rao metric is invariant with respect to the action of $\Diff(M)$ (in contrast to the Otto metric which is not $\Diff(M)$-invariant).

\section{Right-invariant metrics and the left projection}\label{sec:right_inv_metrics}
In this section we study the framework of \autoref{sub:submersions_infinite} in the context of right-invariant Sobolev type Riemannian metrics on $\Diff(M)$ and the left projection $\pi_l$ onto $\Prob(M)$. 
This yields a new class of Riemannian metrics on $\Prob(M)$.
We resolve the infinite-dimensional complications mentioned in \autoref{sub:submersions_infinite}, thereby allowing us to make use of the deep results reviewed in \autoref{Ap:rightinvariant} to prove geodesic completeness.
% To rigorously define new Riemannian structures on  the space of probability densities using this construction 
% and  to make use of deep results on global existence for right-invariant metrics on $\Diff(M)$ we will need  first to address the complications mentioned in \autoref{sub:submersions_infinite}.

Let $G$ be a right-invariant Riemannian metric on $\Diff(M)$ as introduced in  \autoref{Ap:rightinvariant}, i.e., given at the identity by
\begin{equation}\label{eq:right_inv_metricG}
 G_{\operatorname{\id}}(X,Y)=\int_M g\left(A X,  Y\right)\, \vol=\int_M g\left(X,  AY\right)\, \vol\;,
\end{equation}
where $A\colon \Xcal(M)\to \Xcal(M)$ is a positive, elliptic, differential operator, self-adjoint with respect to the $L^2$ inner product on $\Xcal(M)$.
% We want to emphasize here, that the same statement is
% not true for right-invariant metrics and the right projection, where the descending property is an extremely rare property, see the work of \citet{} and \citet{}.

We first address the question of existence of the horizontal bundle.

% We will  show that for right-invariant metrics and the left projection the vertical bundle is complimented within the smooth Fréchet category --  thus giving a positive answer to the first complication sketched in section
% \autoref{sub:submersions_infinite}:
\begin{lemma}\label{lem:hor_exists}
Let $G$ be a right-invariant metric on $\Diff(M)$ of the form~\eqref{eq:right_inv_metricG}.
Then the horizontal bundle with respect to the left projection $\pi_l$ exists in the Fréchet topology as a complement of the vertical bundle~$\on{Ver}^l$. 
It is given by
\begin{align}
 \operatorname{Hor}_{\varphi}
 &=\left\{\left(A^{-1} (\rho \nabla p)\right)\circ\varphi\mid p\in C^{\infty}(M)\right\},
\end{align}
where $\rho=\on{Jac}(\varphi^{-1})$. 
Thus, every vector $X\in T_\varphi\Diff(M)$ has a unique decomposition
$X=X^{\operatorname{Ver}}+X^{\operatorname{Hor}}$ with $X^{\operatorname{Ver}}\in\operatorname{Ver}^l_{\operatorname{\varphi}}(\pi)$ and $X^{\operatorname{Hor}}\in \operatorname{Hor}_{\varphi}(\pi)$.
\end{lemma}

\begin{proof}
Let $h=u\circ\varphi \in T_{\varphi}\Diff(M)$. 
Then $h\in \operatorname{Hor}_{\varphi}(\pi)$ if and only if
\begin{equation}
G_{\varphi}(h,k)=0, \qquad \forall k\in  \operatorname{Ver}^l_{\varphi}(\pi)\;.
\end{equation}
Let $k=v\circ\varphi$ and $\rho = \Jac(\varphi^{-1})$. 
Using the characterization of $\operatorname{Ver}^l_{\varphi}(\pi)$ in \autoref{lem:vertical_bundle} this yields
\begin{equation}\label{eq:hor_char}
G_{\varphi}(h,k)=\int_M g(Au,v) \vol=0, \qquad \forall v\in \Xcal(M) \text{ with } \operatorname{div}(\rho v)=0 
\end{equation}
Consider now the vector field $w=\frac1{\rho}Au$.
The Hodge decomposition for $w$ yields
\begin{equation}
w=\nabla p + \tilde w ,
\end{equation}
with unique components $p\in  C^{\infty}(M)/\RR$ and $\tilde w\in \Xcalvol(M) = \{u\in\Xcal(M)\mid\divv u =0 \}$.
Thus, we can decompose $u$ as
\begin{align}\label{eq:u_Ainv}
u= A^{-1}(\rho\nabla p +\rho \tilde w),
\end{align}
with both $A^{-1}(\rho\nabla p)$ and $A^{-1}(\rho \tilde w)$ in $\Xcal(M)$. 
Plugging \eqref{eq:u_Ainv} into \eqref{eq:hor_char} then yields
\begin{equation}
G_{\varphi}(h,k)=\int_M g(\rho\nabla p +\rho \tilde w,v) \vol = \int_M g(\rho\nabla p,v) \vol +\int_M g(\rho \tilde w,v) \vol \; .
\end{equation}
Using integration by parts, the first term vanishes
\begin{equation}
 \int_M g(\rho\nabla p,v) \vol =\int_M g(\nabla p,\rho v) \vol=-\int_M  p\operatorname{div}(\rho v) \vol=0.
 %-\int_M  f. \operatorname{div}^{\varphi_*\vol}(Y) \varphi_*\vol =0
\end{equation}
Thus, $k=u\circ\varphi$ is horizontal if $u$ is of the form $A^{-1}(\rho\nabla f)$.
It remains to show that if $\tilde w\neq 0$, then $u\circ\varphi$ is not horizontal.
For this, we note that $v=\frac{1}{\rho}\tilde w$ satisfies $\operatorname{div}(\rho v)=0$ and
\begin{equation}
	\int_M g(\rho\tilde w,v)\vol = \norm{\tilde w}_{L^2}^2 .
\end{equation}
This concludes the characterization of the horizontal bundle.
\end{proof}

A consequence of \autoref{lem:hor_exists} is that the analogue of \autoref{prop:descending_metric}  is valid in our infinite-dimensional situation: the Riemannian metric $G$ induces a Riemannian metric on $\Prob(M)$.
To see what the induced metric is, we need to calculate the horizontal lift of a tangent vector $\dot \mu\in T_\mu\Prob(M)$.
To this end, we introduce a field of pseudo differential operators over $P^\infty(M)$ given by
% $L_{\rho}$, which will appear in the formula for the horizontal lift and thus in the formula for the metric $\bar G$:
\begin{equation}\label{Lrho}
L_{\rho}\colon\begin{cases} 
      C^{\infty}(M)/\mathbb R&\longrightarrow C^{\infty}_0(M)  \\
      p&\longmapsto -\on{div}(\rho A^{-1}(\rho\nabla p))\;.\\
   \end{cases}   
\end{equation}
Geometrically, one should think of the field $L_\rho$ as the inverse of a Legendre transform, identifying (the smooth part of) the cotangent bundle $T^* \Prob(M)\simeq P^\infty(M)\times C^\infty(M)/\RR$ with the tangent bundle $T\Prob(M)\simeq P^\infty(M)\times C^\infty_0(M)$.

% We also need the following result concerning the invertibility of $L_{\rho}$.
\begin{lemma}\label{lem:Lrho_inverse}
Let $A$ be a positive, elliptic, differential operator of order $2k+2$, self-adjoint with respect to 
the $L^2$ inner product. 
For any $\rho\in P^\infty(M)$ the pseudo differential operator operator $L_{\rho}$ of order $-2k$ defined in \eqref{Lrho} is an isomorphism.
% where the inverse of $L_{\rho}$ is given by
% \begin{equation}
% \bar A_{\rho}:=L_{\rho}^{-1}: \begin{cases} 
%       C^{\infty}(M)/\mathbb R&\to C^{\infty}(M)/\mathbb R  \\
%       f&\mapsto \Delta^{-1} (\on{div} (\frac{1}{\rho} A \frac{1}{\rho}\nabla (\Delta^{-1} f)))\;.\\
%    \end{cases}   
% \end{equation}
% is invertible.
\end{lemma}
\begin{proof}
Using integration by parts, $L_{\rho}$ is self adjoint since $A^{-1}$ is.
For any $s$ in $\mathbb N \cup \infty$ we can extend $L_{\rho}$ to a bounded linear operator $H^s(M)/\mathbb R\to H^{s+2k}_0(M)$.  
To prove that $L_\rho$ is an elliptic operator, we decompose it in its components
\begin{align}
L_{\rho}=  -\on{div}\circ M_{\rho}\circ A^{-1} \circ M_{\rho} \circ\nabla\;,
\end{align}
where $M_{\rho}$ is the multiplication operator with $\rho$. 
$M_{\rho}$ is elliptic since $\rho(x)>0$ for all $x\in M$.
Thus, $L_{\rho}$ is weakly elliptic as it is a composition of weakly elliptic operators; here one uses the fact that the principal symbol is multiplicative, see \cite[Sect.~4]{LM1989}.
As a next step, we want to determine the kernel of $L_{\rho}$. 
% Therefore we calculate:
\begin{align}
\int_M  L_{\rho}(p) p\, \vol = -\int_M \on{div}(\rho A^{-1}(\rho\nabla p)) p \vol
=\int_M g(A^{-1}(\rho\nabla p)),\rho \nabla p) \vol>0 
\end{align}
for all $p\neq [0] \in H^s(M)/\RR$. 
Here we use that 
\begin{align}
\int_M g(A^{-1}u,u) \vol>0
\end{align}
for all $u\in \Xcal(M)\backslash \{ 0\}$.
Thus $L_{\rho}$ is injective, as it is strictly positive on $H^s(M)/\mathbb R$. 
Since it is Fredholm with index zero it is also surjective.
The isomorphism result is valid for smooth functions due to elliptic regularity, see \cite[Sect.~5]{LM1989}.
\end{proof}

We now obtain an isomorphism between $\on{Hor}_\varphi$ and $T_{\pi_l(\varphi)}\Prob(M)$ analogous to the finite-dimensional case \eqref{sh:sub:eq1}.

% We are now able to calculate the formula for the horizontal lift of a vector field on $\Prob(M)$, which is the remaining ingredient to obtain the induced metric on $\Prob(M)$.
\begin{lemma}\label{hor:lift}
Let $G$ be a right-invariant metric on $\Diff(M)$ of the form~\eqref{eq:right_inv_metricG}.
Then 
\begin{equation}
	T_\varphi\pi_l|_{\on{Hor}_\varphi}\colon \on{Hor}_\varphi \to T_{\pi_l(\varphi)}\Prob(M)
\end{equation}
is an isomorphism.
% Given a tangent vector $\dot \rho\vol \in T_{\pi_l(\varphi)}\Prob(M)$, 
The inverse is given by
\begin{equation}
T_{\pi_l(\varphi)}\Prob(M) \ni \dot \rho\vol \mapsto A^{-1}(\rho \nabla p)\circ\varphi\in\on{Hor}_\varphi,
\end{equation}
where %$f\in C^{\infty}(M)/\mathbb R$ is given by
\begin{equation}
p=L_{\rho}^{-1}(\dot \rho)\;.
\end{equation}
\end{lemma}
\begin{proof}
The horizontal lift of a tangent vector $\dot \mu \in T_{\mu}\Prob(M) $ is the unique horizontal vector field $u$ such that
\begin{align}
T_{\varphi}\pi_l (u\circ\varphi)= \dot \mu
\end{align}
where $\varphi$ is some diffeomorphism with $\pi(\varphi)=\mu$.
Writing $\mu=\rho \vol$ and using the characterization of the horizontal bundle and the formula for $T\pi$ this yields the equation
\begin{align}
\dot \mu = -\on{div}^{\varphi_*\vol}(u).\varphi_*\vol= -\on{div}^{\varphi_*\vol}( A^{-1}(\rho\nabla p)).\varphi_*\vol = -\on{div}(\rho A^{-1}\rho\nabla p)\vol
\end{align}
Dividing by $\vol$ the above lifting equation can be rewritten as
\begin{align}
\frac{\dot \mu}{\vol} = -\on{div}(\rho A^{-1}(\rho\nabla p))=L_{\rho}(p).
\end{align}
Applying the inverse of $L_{\rho}$ to the above equation yields the desired result.
\end{proof}

% \subsection{The induced metric on \texorpdfstring{$\operatorname{Prob}(M)$}{Prob(M)} }
Using \autoref{hor:lift} we obtain the formula for the induced metric on $\Prob(M)$. 

\begin{proposition}\label{pro:descending_metric_formula}
Let $G$ be a right-invariant metric on $\Diff(M)$ of the form~\eqref{eq:right_inv_metricG} with inertia operator $A$ as in \autoref{lem:Lrho_inverse} of order $2k+2$.
Then the induced  metric on $\Prob(M)$ with respect to the projection $\pi_l$ is given by
 \begin{align}\label{met:dens}
\bar G_{\mu}(\dot \mu,\dot \mu)&= %\bar G_{\rho\vol}(\dot \rho \vol,\dot \rho \vol)=
\int_M  (\bar A_\rho \dot \rho)\, \dot\mu ,%=\int_M  (\Delta^{-1} \on{div} \frac{1}{\rho} A \frac{1}{\rho}\nabla \Delta^{-1} \dot \rho)\; \dot\rho \;\vol. 
\end{align}
where $\bar A_\rho \coloneqq L_{\rho}^{-1}$, $\mu=\rho\vol$ and $\dot\mu = \dot\rho \vol$.
% The metric $\bar G$ is of order $l$ as the inertia 
The pseudo-differential operator $\bar A_{\rho}$ is of order $2k$, so $\bar G$ is of order $k$.
\end{proposition}

\begin{proof}
From \autoref{hor:lift} for the horizontal lift of a tangent vector we get
\begin{align}
\bar G_{\rho\vol}(\dot \rho \vol,\dot \rho \vol)&= G_\id( \rho \nabla \bar A_{\rho}(\dot \rho), \rho \nabla \bar A_{\rho}(\dot \rho)) = \int_M g\left(   \rho \nabla \bar A_{\rho}(\dot \rho), A^{-1}\left( \rho \nabla \bar A_{\rho}(\dot \rho)  \right)\right )\vol.
\end{align}
% where $\varphi$ is any diffeomorphism in the fiber of $\rho\vol$.
Using  integration by parts and that $A$ is self-adjoint we obtain
\begin{align}
G_{\mu}(\dot \mu, \dot \mu)&= \int_M g\left(  \on{div}  \rho A^{-1} \rho \nabla \bar A_{\rho}(\dot \rho),    \bar A_{\rho}(\dot \rho) \right )\vol
\end{align}
Since $\bar A_{\rho}^{-1}=L_{\rho}=\on{div}  \rho A^{-1} \rho \nabla$ this equals
 \begin{align}
G_{\mu}(\dot \mu, \dot \mu)&= \int_M g\left(  \bar A_{\rho}^{-1}\bar A_{\rho}(\dot \rho),   \bar A_{\rho}(\dot \rho) \right )\;\vol\\
&=\int_M g\left(  \dot \rho,   \bar A_{\rho}(\dot \rho) \right )\;\vol\;.
\end{align}
The order of the pseudodifferential operator $\bar A_\rho$ follows by counting derivatives.
\end{proof}

We are now ready to give the main result of the paper, concering geodesic completeness for Riemannian metric on $\Prob(M)$. 
(The result stated here covers \autoref{thm:main result} in \autoref{sub:summary}.)

% Using \autoref{thm:wellposednessDiff} and the result that $\pi_l$ is a Riemannian submersion we immediately obtain our main result concerning existence of geodesics on the space of probability densities:
\begin{theorem}[Main Result]\label{thm:main:result}
Let $\bar G$ be a Riemannian metric on $\Prob(M)$ of the form in \autoref{pro:descending_metric_formula}.
The geodesic equation expressed in Hamiltonian coordinates $(\rho,p)\in P^\infty(M)\times C^\infty(M)/\RR$ is then given by
% on the space of probability densities, given by: 
% $$\bar G_{\mu}(\dot \rho \vol,\dot \rho \vol)=\int_M  (\bar A_\rho \dot \rho)\; \dot\rho \;\vol, \qquad\text{ with }\qquad \bar A_\rho= -\left(\on{div} \rho A^{-1} \rho\nabla \right)^{-1}\;,$$
% for  $A$  a positive, elliptic, differential operator of order $2l+2$, that is self-adjoint with respect to 
% the $L^2$ inner product. The geodesic equations of the metric 
% $\bar G$ on $\Prob(M)$ are then given by:
\begin{equation}\label{geodesic:equation}
\begin{aligned}
\rho_t &= L_{\rho}(p) = -\on{div}(\rho A^{-1}(\rho\nabla p)),\\
p_t&=-\nabla p\cdot A^{-1}(\rho \nabla p)
\end{aligned}
\end{equation}
We have the following existence results for solutions to this equation:
\begin{itemize}
\item
If the order of $\bar G$ satisfies $k\geq 0$  then, for any initial conditions $(\rho_0,p_0) \in P^\infty(M)\times C^\infty(M)/\RR$, there exists a unique solution 
 \begin{equation}
  (\rho,p)\in C^{\infty}(J,T\Prob(M)) \quad\text{with}\quad \rho(0)=\rho_0 \quad\text{and}\quad p(0)=p_0,
 \end{equation}
defined on a non-empty, non-extendable existence interval $J$.
For any $t\in J$, $(\rho(t),p(t))$ depends smoothly on $(\rho_0,p_0)$.
% of the geodesic equation of the metric $\bar G$ with initial condition $(\vol,\al_0)$, defined on the maximal interval of existence $J$ and the solution depends smoothly on the initial data.
\item 
If $k> \frac{d}{2}$ then, for any initial conditions $(\rho_0,p_0)\in P^\infty(M)\times C^\infty(M)/\RR$, the existence interval $J$ is equal to $\mathbb R$.
That is, the infinite-dimensional Riemannian manifold $(\Prob(M),\bar G)$ is geodesically complete.
\end{itemize}
\end{theorem}

% \begin{remark}
% Note, that the derivation of the geodesic equation implicitly deals with the second complication described in \autoref{sub:submersions_infinite}.
% \end{remark}

\begin{remark}
To obtain the form of the geodesic equation resembling Otto's formulation of optimal transport, as presented in \autoref{sub:summary}, we introduce the vector field $u\in\Xcal(M)$ via
\begin{equation}
	A u = \rho\nabla p .
	\end{equation}
% where $p$ is implicitly defined via the continuity equation
% \begin{equation}
% 	\rho_t + \operatorname{div}(\rho u) = 0.
% \end{equation} 
Then the geodesic equation \eqref{geodesic:equation} becomes
\begin{equation}
	\begin{split}
		&\rho_t + \divv(\rho u) = 0, \\
		&p_t + u\cdot \nabla p = 0.
	\end{split}
\end{equation}
\end{remark}

\begin{proof}
To derive the formula for the geodesic equation, we consider the EPDiff equation of the $G$-metric on $\Diff(M)$ for horizontal initial conditions. 
% This will show at the same time the conservation of horizontality and calculate the  geodesic equation on the space of probability densities.
Recaling from \autoref{thm:wellposednessDiff} the geodesic equation of a right-invariant metric on $\Diff(M)$ we have
\begin{align}
\varphi_t &= u\circ\varphi \label{eq:reconstruction}\\
m_t&=-\left\{\nabla_u m+(\nabla u)^\top m+(\mathrm{div}\hspace{0,1cm}u)m\right\}\;, \label{eq:epdiffinproof}
\end{align}
where $m = Au$ is the \emph{momentum variable}.
For any curve $u(t)$ in $\Xcal(M)$ it directly follows by differentiating the relation $\rho(t) = \Jac(\varphi(t)^{-1})$ that $\rho$ fulfills the first part of \eqref{geodesic:equation}.
This also proves that $\rho(t)\in P^\infty(M)$ as long as $\varphi(t)\in\Diff(M)$.

 % and $u=\varphi_t\circ\varphi^{-1}$ is the right trivialized variable.
Let now $\rho(t)$ be a path on $P^\infty(M)$ and let
$p=\bar A_{\rho} \rho_t$.
Using \autoref{hor:lift} we obtain the horizontal vector field $u$ corresponding to $\rho_t$ as
\begin{align}
u = A^{-1}(\rho \nabla p).
\end{align}
Substituting this in the EPDiff equation gives
\begin{align}
u &= A^{-1}(\rho \nabla p),\quad m=\rho \nabla p\\
m_t&=-\left\{\nabla_u m+(\nabla u)^\top m+(\mathrm{div}\hspace{0,1cm}u)m\right\}\,.
\label{eq:epdiff_subst}
\end{align}
Expanding the left side of \eqref{eq:epdiff_subst} yields
\begin{align}
m_t= \rho_t \nabla p+\rho \nabla p_t\;.
\end{align}
For the divergence term in the right hand side of \eqref{eq:epdiff_subst} we have
\begin{align}
 (\mathrm{div}\hspace{0,1cm}u) m&	= (\mathrm{div}\hspace{0,1cm}u) \rho \nabla p =
 \mathrm{div}(\rho u) \nabla p -(u \cdot \nabla \rho) \nabla p .
 \end{align}
 Using that 
 $$u= A^{-1}(\rho \nabla p)=A^{-1}(\rho \nabla (\bar A_{\rho} \rho_t )$$
 we have 
 \begin{align}
 \mathrm{div}(\rho u) \nabla p =
\mathrm{div}\left(\rho A^{-1}(\rho \nabla \bar A_{\rho} \rho_t \right)  \nabla p=
 -\bar A^{-1}_{\rho}\bar A_{\rho}(\rho_t) \nabla p=-\rho_t \nabla p\;.
 \end{align}
Thus \eqref{eq:epdiff_subst} simplifies to
\begin{align}
 \rho \nabla p_t&=-\nabla_u m-(\nabla u)^\top m+ (u \cdot \nabla \rho) \nabla p.
\end{align}
The right hand side can now be rewritten as
\begin{align}
% -\nabla_u m-(\nabla u)^\top m + (u \cdot \nabla \rho) \nabla p =
-\nabla_u (\rho \nabla p)-(\nabla u)^\top (\rho \nabla p)+ (u \cdot\nabla \rho) \nabla p=-\rho \nabla ( u\cdot\nabla p)\;.
\end{align}
Equation \eqref{eq:epdiff_subst} can thereby be written
\begin{equation}\label{eq:final_eq_for_p}
	\rho\nabla p_t = -\rho \nabla (u\cdot \nabla p).
\end{equation}
Since $\rho>0$ we get that $p\in C^\infty(M)/\RR$ fulfills the second part of \eqref{geodesic:equation} if and only if $m=\rho\nabla p$ fulfills the EPDiff equation \eqref{eq:epdiff_subst}.
The same calculation also shows explicitly that if $m(t)$ is a solution to the EPDiff equation, and $m(0)$ is of the form $\rho\nabla p$, then $m(t)$ is of the same form for any $t$.
That is, initially horizontal geodesics remain horizontal.

To prove the well posedness results, we need, for any choice of initial data $(\rho_0,p_0)$, to construct a solution $(\rho,p) \in C^{\infty}(J,P^\infty(M),C^\infty(M)/\RR)$ to equation \eqref{geodesic:equation} with $\rho(0)=\rho_0$ and $p(0) = p_0$, and to show that this solution is unique.
To this end, let $\varphi_0$ be an arbitrary diffeomorphism such that $\varphi_{0*}\vol = \rho_0\vol$ (such a diffeomorphism always exist by \autoref{thm:principal_bundle_diff_dens}).
Let $\dot\rho_0 = L_{\rho_0}p_0$.
Using \autoref{hor:lift} we can lift $(\rho_0,\dot\rho_0)$ to a unique horizontal vector $u_0\circ \varphi_0 \in T_{\varphi_0}\Diff(M)$.
Using  \autoref{thm:wellposednessDiff} with $A$ of order $2k+2\geq 2$ we obtain a unique solution $(\varphi,u)$ with $\varphi(0) = \varphi_0$ and $u(0) = \dot\varphi_0\circ\varphi_0^{-1}$ on a non-empty maximal existence interval $J$ containing $0$.
Since $\dot\varphi(0) \in \on{Hor}_{\varphi(0)}$ it follows from the calculation above that $\dot\varphi(t) \in \on{Hor}_{\varphi(t)}$ for every $t\in J$.
Thus, the momentum variable $m(t) \coloneqq Au(t)$ is of the form $m(t) = \rho(t)\nabla p(t)$ where $\rho(t) = \Jac(\varphi(t)^{-1})$.
Notice that $\rho(t)\in P^\infty(M)$ is unique for any $t\in J$ since $\varphi(t)$ is unique and the choice of representative $\varphi_0$ for $\rho_0$ does not affect $\rho(t)$ due to right-invariance of the EPDiff equation.
It follows from \autoref{lem:Lrho_inverse} and \autoref{hor:lift} that $p(t) \in C^\infty(M)/\RR$ is unique for any $t\in J$.
By the calculations leading from \eqref{eq:reconstruction} and \eqref{eq:epdiffinproof} to \eqref{eq:final_eq_for_p} we see that $(\rho(t),p(t))$ fulfill \eqref{geodesic:equation}.
That $(\rho(t),p(t))$ for $t\in J$ depends smoothly on the initial conditions follows from \autoref{thm:wellposednessDiff} and the smooth principal bundle result in \autoref{thm:principal_bundle_diff_dens}.
If $k>\frac{d}{2}$ it follows from \autoref{thm:wellposednessDiff} that $J=\RR$.
This proves the results.
%  let $\vol, \alpha_0$ be given initial conditions. Using \autoref{hor:lift} we can lift it to a horizontal vector field $v$ on $\Diff(M)$. Using  \autoref{thm:wellposednessDiff} we obtain the local well-posedness of the geodesic equation for $l\geq 0$. Since geodesics with horizontal initial velocity stays horizontal for all time and since horizontal geodesics on $\Diff(M)$ project onto geodesics on 
% $(\Prob(M),G)$, cf. \autoref{Riemanniasubmersions}, this shows the local well-posedness on the space of probability densities. The exact same argument  shows the geodesic completeness, provided $l>\frac{d}{2}$.
\end{proof}
\subsection{Ill-posedness in the Sobolev category}
Recall from \autoref{sec:right_inv_metrics} that the strong form of the geodesic equation with respect to a right-invariant metric $G$ of the form \eqref{eq:Gmet_DiffM} is well-posed on the Sobolev completion $\mathcal{D}^s(M)$ (if $s$ is large enough).
Indeed, although there appears to be a loss of derivatives in the EPDiff equation \eqref{eq:epdiff2}, when formulated in the Lagrangian coordinates $(\varphi,\dot\varphi) \in T\mathcal{D}^s(M)$ there is no loss of derivatives.
However, for the geodesic equation \eqref{geodesic:equation} examined in this paper \emph{there is} a loss of derivatives.
Indeed, it follows directly from \eqref{geodesic:equation} that there is a loss of one derivative in the right hand side of of both $\rho_t$ and $p_t$.
Thus, the strong form \eqref{geodesic:equation} of the geodesic equation is not well-posed in the Sobolev $H^s$ category.
The new system \eqref{geodesic:equation} thereby constitutes an interesting example of a strongly formulated PDE with loss of derivatives, yet well-posedness in the smooth category (like the heat equation), but non-parabolic and fully non-linear (unlike the heat equation).

From a geometrical point of view, the reason behind the appearance of the loss of derivatives when going from diffeomorphisms to probability densities is that the left projection \eqref{projection_left} is not smooth (only continuous).
We expect that there is a suitable weak formulation of \eqref{geodesic:equation} that is well-posed in the Sobolev category.
Indeed, if $\gamma(t)$ is a horizontal geodesic on $\mathcal{D}^{s}(M)$, then $\bar\gamma(t) = \pi^s_l(\gamma(t))$ is a candidate curve on $\Prob^{s-1}(M)$.
Since $\gamma(t)$ is smooth and $\pi^{s}_l$ is a continuous mapping, $\bar\gamma(t)$ is a continuous curve on $\Prob^{s-1}(M)$.
However, due to the lack of smoothness of $\pi^{s}_l$, the curve $\bar\gamma(t)$ is not smooth (not even $C^1$), so it can only fulfill a weak version of \eqref{geodesic:equation}.
A detailed study of these matters is a future topic.
 
% In this part we want to show that the geodesic equations are ill-posed in the Sobolev category. 
% Therefore we consider the simplest example, i.e., $M=S^1$. 
% Then the geodesic equations reads as:
% \begin{equation}\label{geodesic:equation:dim1}
% \begin{aligned}
% \rho_t &= -(\rho A^{-1}(\rho p_x))_x,\\
% p_t&=-p_x.A^{-1}(\rho p_x)
% \end{aligned}
% \end{equation}
% We can rewrite the first equation to obtain
% \begin{equation}
% \rho_t = -\rho_x A^{-1}(\rho p_x)-\rho A^{-1}(\rho_x p_x+\rho p_{xx})\;,
% \end{equation}
% or equivalently
% \begin{equation}
% A^{-1}(\rho p_x) \rho_x = - \rho_t-\rho A^{-1}(\rho_x p_x+\rho p_{xx})\;.
% \end{equation}
% From here it is easy to see that the left hand side is only in $H^{q-1}$, 
% while the right hand side is in $H^q$ -- for $A$ of order greater or equal then two. This shows the ill-posedness in the Sobolev category.\todo[inline,author=KM]{Write something about finding weak solutions in the $H^s$ category: we sort of have them already.}

\section*{Acknowledgments}
The authors would like to thank Martins Bruveris and Francois-Xavier Vialard for fruitful discussions.
This project has received funding from the European Union's Horizon 2020 research and innovation programme under grant agreement No 661482, and from the Swedish Foundation for Strategic Research under grant agreement ICA12-0052.

\bibliographystyle{amsplainnat} % (plain, alpha, abbrv, acm, ieee, abbrvnat)
\bibliography{density_matching}

\end{document}